\documentclass[11pt, oneside, reqno]{amsart}
\usepackage{fullpage}
\usepackage{amssymb,latexsym,amsmath,verbatim,layout,hyperref,amsthm}  
\numberwithin{equation}{section}
\usepackage[T1]{fontenc}
\usepackage[light,math]{anttor}
\usepackage {mathrsfs}                
  
\theoremstyle{mytheoremstyle}

\newtheorem{thm}{Theorem}[section]
\newtheorem*{thm*}{Theorem}
\newtheorem{lem}[thm]{Lemma}
\newtheorem*{lem*}{Lemma}
\newtheorem{prop}[thm]{Proposition}
\newtheorem*{prop*}{Proposition}

\newtheorem*{cor}{Corollary}

\theoremstyle{definition}

\theoremstyle{remark}
\newtheorem*{rmk}{Remark}

\usepackage{hyperref}

\renewcommand{\vec}[1]{\mathbf{#1}}

\newcommand{\R}{\mathbb{R}}
\renewcommand{\P}{\mathbb{P}}
\newcommand{\E}{\mathbb{E}}

\title{Invariant measures for the Stochastic Navier-Stokes equation on a 2D rotating sphere with stable L\'evy noise}
\author{Leanne Dong}
\date{
School of Mathematics and Statistics, The University of Sydney.\,\,	\today
}
\begin{document}
	\maketitle
\begin{abstract}
Building upon the well-posedness results in \cite{snse1}, in this note we prove the existence of invariant measures for the stochastic Navier-Stokes equations with stable L\'evy noise. The crux of our proof relies on the assumption of finite dimensional L\'evy noise.
\end{abstract}	

\section{Invariant measures}\label{sec:im}
    
In this note we are concerned with the existence of an invariant measure of the solution $u$ to the abstract equation (\ref{asnse4}). Let $A$ be the Stokes operator in $H$. Let $B$ is the bilinear operator. Let $C$ be the Coriolis operator in $H$ and $G$ is a bounded linear operator. (See \cite{snse1} for rigorous definitions of these operators)
\begin{align}\label{asnse4}
	du(t)+Au(t)+B(u(t),u(t))+\mathbf{C}u=fdt+GdL(t),\quad u(0)=u_0,
\end{align}

In our earlier work \cite{snse1} we proved 
\begin{itemize}
	\item Theorem \ref{t1u} on the existence and uniqueness of a weak solution.
	\item Theorem \ref{t2u} on the continuous dependence on initial data.
	\item Theorem \ref{t4} on the existence of a strong solution.
\end{itemize}

\begin{thm}\label{t1u}
	Suppose that $\alpha\ge 0$, $z\in L^4_{\text{loc}}([0,\infty);\mathbb{L}^4(\mathbb{S}^2)\cap H)$, $ v_0\in H$ and $ f\in V'$. Then there exists $\P$-a.s. a unique solution $ u\in D([0,\infty);H)\cap L^2_{\text{loc}}([0,\infty);V)$ of equation (\ref{asnse4}). In particular, if 
	\[\sum_{l=1}^\infty|\sigma_l|^{\beta}\lambda^{\beta/2}_l<\infty\,,\]
	then the theorem holds. 
\end{thm}

\begin{thm}\label{t2u}
Assume that,
\begin{equation*}
	 u^0_n\to  u\quad\text{in}\quad H
\end{equation*}
and for some $T>0$,
\begin{align}
	 z_n\to  z\quad\text{in}\quad L^4([0,T];\mathbb{L}^4(\mathbb{S}^2)\cap H)\qquad  f_n\to  f\quad\text{in}\quad L^2(0,T;V').
\end{align}
\end{thm}
\begin{thm}\label{t4}
	Assume that $\alpha\ge 0$, $z\in L^4_{\text{loc}}([0,\infty);\mathbb{L}^4(\mathbb{S}^2)\cap H)$, $ f\in H$ and $ v_0\in H.$ Then, there exists  $\P$-a.s. unique solution of (\ref{asnse4}) in the space $D(0,T;H)\cap L^2(0,T;V).$ which belongs to $D(\epsilon,T; V)\cap L^2_{\text{loc}}(\epsilon,T;D(A))$ for all $\epsilon>0.$ and $T>0.$ Moreover, if $ v_0\in V$, then $ u\in D(0,T; V)\cap L^2_{\text{loc}}(0,T;D(A))$ for all $T>0$, $\omega\in\Omega$. Moreover, if
	\[\sum_{l=1}^\infty|\sigma_l|^{\beta}\lambda^{\beta/2}_l<\infty\,,\]
	then the theorem holds.
\end{thm}

It is well known (see for instance Chapter 9 of \cite{MR2356959}) that strong solution implies a weak solution, and the weak solution is equivalent to a mild solution. Hence the three concepts of solutions are equivalent.  With the aid to these results, our main aim in this section is to study the large time behaviour of $u$, that is, the law $\mathcal{L}(u(t,x))$ as $t\to\infty.$ In particular, we prove (\ref{asnse4}) admits at least one invariant measure
Here we consider a general cadlag Markov process,
\begin{align}
    (\Omega,\{\mathcal{F}^0_t\}_{t\ge 0},\mathcal{F},\{u^x_t\}_{t\ge 0},(\P_x)_{x\in H})
\end{align}
whose transition probability is denoted by $\{P(t,x,dy\}_{t\ge 0}$, where $\Omega :=D([0,\infty);H)$ is the space of the c\`adl\`ag function from $[0,\infty)$ to $H$ equipped with the Skorokhod topology, $\mathcal{F}^0_t=\sigma\{u_s,0\le s\le t\}$ is the natural filtration.
Now denote (resp.) by $C_b(H)$, $B_b(H)$ the space of bounded continuous functions and the space of bounded borel measurable functions on $H.$ That is,
\begin{align}
    C_b(H) &:=\{\varphi:H\to\R: \varphi\,\,\text{is continuous and bounded}\},\\
    B_b(H) &:=\{\varphi:H\to\R: \varphi\,\,\text{is bounded and borel measurable}\}.
\end{align}
For all $\varphi\in B_b(H)$, define
\begin{align*}
    P_t \varphi(x)=\int_H \varphi(y)P(t,x,dy),\quad\forall\,\,t\ge 0,x\in H.
\end{align*}
For any $t\ge 0$, $P_t$ is said to be Feller if
\begin{align}\label{feller}
    \varphi\in C_b(H)\to P_t\varphi\in C_b(H),\quad\forall\quad t\ge 0.
\end{align}
$P_t$ is said to be strong Feller if (\ref{feller}) holds for a larger class of function: $\varphi\in B_b(H).$
Moreover, $P_t$ is said to be \emph{irreducible} in $H$, if $P_t 1_{A}(x)=P_x(t,A)>0$ for any $x\in H$ and any non-empty open subset $A$ of $H$. If $P_t$ is irreducible then any invariant measure $\mu$ is full, that is, one has $\mu(B(x,r))>0$ for any ball $B(x,r)$ of center $x\in H$ and radius $r$. Indeed, it follows from the definition of invariant measure that
\begin{align*}
    \mu(B(x,r))&=\int_H P_t 1_{A}(x)\mu(dx)>0.
\end{align*}

The main theorem proved in this section is Theorem \ref{existenceim} which states below.  
\begin{thm}\label{existenceim}
Assume additionally, that there exists $m>1$ such that $\sigma_l=0$ for all $l\ge m$. Then 
    the solution $u$ to (\ref{asnse4}) admits at least one invariant measure.
\end{thm}
    
We claim that the SNSE (\ref{asnse4}) has an invariant measure. The key to proving this is to use the Krylov-Bogolyubov Theorem, which guarantees the existence of invariant measures for certain well-defined maps defined on some well-defined space. More precisely, the theorem states that,
\begin{thm}[Krylov-Bogolyubov]
    Assume $(P_t,t\ge 0)$ is a \textbf{Feller} semigroup. If there exists a point $x\in H$ for which the family of probability measure $\{\mu_t(x,\cdot)\}_{t\ge 0}$ is uniformly \textbf{tight}, that is, there exists a compact set $K_{\epsilon}\subset H$ such that $\mu(K_{\epsilon})\ge 1-\epsilon$ for any $\mu\in\Lambda$ on $(H,\mathcal{B}(H))$ then there exists at least one invariant measure.
\end{thm}

\begin{cor}\label{kbcor}
    If for some $\nu\in\mathcal{P}$ and some sequence $T_n\uparrow +\infty$ the sequence $\{P^*_{T_n}\nu\}$ is tight, then there exists an invariant measure for $(P_t,t\ge 0).$
\end{cor}
\vspace{1cm}
We shall remark that there are various versions of Krylov-Bogolyubov theorem which conveys the same idea. All that required to be proved are Feller, Markov property of the solution $v$ (and so $u$) and convergence of the family of probability measures $\{\mu_t, t\ge 0\}$ in $H$. This is comparable to the concept of weak convergence of distribution in finite dimension (equivalence to weak convergence of r.v.). However, in infinite dimension, the convergence of distribution is more involved. Hence extra conditions are needed besides the convergence of finite-dimensional distributions.
   
Note that,  it is known that tightness is a necessary condition to prove convergence of probability measure, especially when measure space is infinite dimensional. In this sense the two statements of the theorem is equivalent.

The following inequalities would be used quite often.
\begin{align*}
    |\hat{ A}^{\sigma}e^{-\hat{ A}t}|&\le C(\sigma)t^{-\sigma},\quad\forall\,\,\sigma>0,
\end{align*}

\begin{align}\label{bua}
    |B( u)|_V= \langle  A^{\frac{1}{2}}B( u), A^{\frac{1}{2}}B( u)\rangle =\langle  AB( u),B( u)\rangle&=| A^{\frac{1}{2}}B(u)|\\
    &\le | u|| u|^{\frac{1}{2}}_V| A u|^{\frac{1}{2}},
\end{align}

\begin{align*}
    |B( u)-B( v)|\le C(| u|^2_V+| u|^2_V| v|+| v|^2_V),
\end{align*}

\begin{align*}
    |B( u)|\le C| u|^{\frac{1}{2}}_V| u|_V|A u|^{\frac{1}{2}}=C| u|^{\frac{3}{2}}_V|A u|^{\frac{1}{2}}.
\end{align*}

\subsection{Transition Semigroup}\label{markovfellersoln}
Let us denote by $u(\cdot,x)$ the solution of (\ref{asnse4}). We set 
\begin{align*}
    P_t f(x)=\E f(u(t,x)),\quad f\in B_b(H), t\ge 0, x\in H.
\end{align*}
It follows from uniqueness and time homogeneity of $L$ that the following relation holds,
\begin{align*}
    P_t\circ P_s=P_{t+s}.
\end{align*}
Recall from our Theorem \ref{t4} that, we have proved there exists a unique strong solution to (\ref{asnse4}) with mild form in the space $D([0,T;H)\cap L^2([0,T];V)$ which belongs to $D([h,T];V)\cap L^2_{\text{loc}}([h,T];D(A))$ for all $h>0$, $T>0$ for every initial condition $ u_0\in H$, $\omega\in\Omega.$ Moreover, if $ u_0\in V$, then $ u\in D([h,T;V])\cap L^2_{\text{loc}}([h,T];D(A)) $ for all $h>0$ and $T>0.$ The solution depends continuously on initial data $ x.$ 

Let $u(t;x)$ be the solution at $t$ starting from $x$ at time 0. 
Now suppose we have two solutions, resp. $ u_n$ and $u$ of (\ref{asnse4}) started at $ \xi_n$ and $ \xi$, if the conditions in Theorem \ref{t2u} satisfied, then it follows that $ u_n(t)\to  u(t)$ a.s. for any $t.$ Therefore, $f( u_n(t))\to f( u(t))$ as $f$ is continuous. Thus, invoke Lesbesgue Dominated Convergence Theorem, one has
\begin{align}
    \E f( u_n(t))\to\E f( u(t)).
\end{align}
Whence the equation (\ref{asnse4}) defines a Feller Markov process. 
 Then we can define the operator $P_t: C_b(H)\to C_b(H)$ by 
\begin{align*}
    P_t f=\E f(u(t;x)),
\end{align*}
and $P_t$ is said to be a Feller semigroup.
\begin{lem}
    The equation (\ref{asnse4}) defines a Markov process in the sense that
    \begin{align}\label{markovian}
        \E[  f ( u^{ x}_{t+s})|\mathcal{F}_t]=P_s(  f )(u^{ x}_t),
    \end{align}
for all $t,s>0$, $  f \in C_b(H)$, where $ u^{ x}_t$ denotes\footnote{The notation $u^{ x}_t$ is used interchangeably with $u(t;x)$.} the solution to (\ref{asnse4}) over $[0,\infty]$ starting from the point $ u(0)=x$, $\mathcal{F}_t$ denotes the sigma-algebra generated by $L(\tau)$ for $\tau\le t.$
\end{lem}
 
By uniqueness,
\begin{align*}
     u^{ x}_{t+s}= u^{ u^{ x}_t}_{t,t+s},\quad (\P-a.s.),
\end{align*}
where $( u^{\eta}_{t_0,t})_{t\ge t_0}$ denotes the \textbf{unique} solution on the time interval $[t_0,\infty)$, with the $\mathcal{F}_{t_0}$-measurable intial condition $ u^{\eta}_{t_0,t_0}=\eta$. To prove (\ref{markovian}), it suffices to prove
\begin{align*}
    \E[  f ( u^{\eta}_{t,t+s})|\mathcal{F}_t]=P_s(  f )(\eta),
\end{align*}
for every $H$-valued $\mathcal{F}_t$-measurable r.v. $\eta.$ Note that (\ref{markovian}) holds for all $  f \in C_b(H)$, holds for $  f =1_{\Gamma}$ where $\Gamma$ is an arbitrary Borel set of $H$ and consequently for all $f\in B_b(H).$  Without loss of generality, Let us assume $f\in C_b(H)$. We know that, if $\eta=\eta_i$ $\P$ a.s., then the r.v. $u(t+s,t,\eta_i)$ is independent to $\mathcal{F}_t$ and therefore
\begin{align*}
    \E(f( u(t+s,t,\eta_i))|\mathcal{F}_t)=\E f( u(t+s,t,\eta_i))=P_{t,t+s} f(\eta_i)=P_s f(\eta_i),\quad\P\,\,\text{a.s.}\quad.
\end{align*}
It suffices to prove (\ref{markovian}) holds for every r.v. $\eta$ of the form
\begin{align*} \eta=\sum^N_{i=1}\eta^{(i)}1_{\Gamma^{(i)}},
\end{align*} where $\eta^{(i)}\in H$ and $\Gamma^{(i)}\subset\mathcal{F}_t$ is a partition of $\Omega$, $\eta_i$ are elements of $H.$ Then
\begin{align*}
     u(t+s,t,\eta_i)=\sum^N_{i=1} u(t+s,t,\eta_i)1_{\Gamma_i},\quad\P\,\,\text{a.s.}\quad.
\end{align*}
Hence,
\begin{align*}
    \E(f( u(t+s,t,\eta))|\mathcal{F}_{t})=\sum^N_{i=1}\E(f( u(t+s,t,\eta_i))1_{\Gamma_i}|\mathcal{F}_t)\quad\P-\text{a.s.}\quad.
\end{align*}
Take into account the r.v. $ u(t+s,t,\eta_i)$ independent to $\mathcal{F}_t$ and $1_{\Gamma_i}$ are $\mathcal{F}_t$ measurable, $i=1,\cdots, k$, one deduces that
\begin{align*}
    \E[f( u(t+s,t,\eta))|\mathcal{F}_t]&=\sum^N_{i=1} P_s f(\eta_i)1_{\Gamma_i}=P_s f(\eta),\quad\P-\text{a.s.}\quad,
\end{align*}
and so (\ref{asnse4}) defines a Markov process in the above sense for all $f\in C_b(H)$. Now, let $u(t;\eta)$ be the solution of the SNSE (\ref{asnse4}) with initial condiction $\eta\in H$. 

Let $(P_t, t\ge 0)$ be the Markov Feller semigroup on $C_b(H)$ associated to the SNSE (\ref{asnse4}) defined as
\begin{align}\label{markovsemi}
    P_t f(\eta)=\E[f(u(t;\eta))]=\int_H f(y)P(t,y)dy=\int_H f(y)\mu_{t,s}(dy),\quad f\in C_b(H),
\end{align}
where $P(t,x,dy)$ is the transition probability of $u(t;\eta)$ and $\mu_{t,x}(dy)$ is the law of $u(t;\eta).$ From (\ref{markovsemi}), we have
\begin{align*}
    P_t f(x)=(f,\mu_{t,x})=(P_t f,\mu),
\end{align*}
where $\mu$ is the law of the initial data $\eta\in H$. Thus it follows from above that $\mu_{t,\eta}=P^*_t\mu$. If
\begin{align*}
    P^*_t\mu=\mu\quad\forall\quad t\ge 0,
\end{align*}
then a probability measure $\mu$ on $H$ is said to be an invariant measure.
\subsection{The proof of tightness }

We proceed the claim of tightness by first proving the following a priori estimate. The main difficulty is overcome by introducing a simplified auxiliary Ornstein Uhblenbeck process, which enables us to use the classical arguments in the spirit of p.51-150 \cite{SFS08}. To prove existence of invariant measures for (\ref{asnse4}), we write the problem in a slightly different form.

Let $H$, $A: D(A)\subset H\to H$, $V=D(A^{1/2})=D(\widehat{A}^{1/2})$ and $B:V\times V\to V', \,\,C$ be spaces and operators introduced in the previous section. Suppose that there exists a constant $c_B>0$ such that
\begin{align}\label{buvw0}
    \langle B(u,v),w\rangle=|b(u,v,w)|&\le c_B|u|^{1/2}|u|^{1/2}_V|v|^{1/2}|v|^{1/2}_V|w|_V,\quad\forall\quad u,v,z\in V,
\end{align}

\begin{align*}
    \langle B(u,v),v\rangle\le c_B|u|^{1/2}|Au|^{1/2}|v|_V|z|,
\end{align*}
for all $u\in D(A)$, $v\in V$ and $z\in H.$

In order to prove there exists at least one invariant measure, we use standard method in the spirit of Chapter 15 in \cite{MR1417491}.  However, the analysis of Navier-Stokes equations with additive noise in our case requires some non-trivial consideration, as pointed out in \cite{MR1305587},. In particular, a critical question arises when analyzing the estimate $\frac{d^+}{dt}|v(t)|^2$,  the usual estimates for the nonlinear term $b(v(t),z(t),v(t))$ yields a term $|v(t)|^2 |z(t)|^4_4,$ so we were not able to deduce any bound in $H$ for $|v(t)|^2$ under classical lines. Nevertheless, in light of the method developed in Crauel and Flandoli \cite{MR1305587}, via the usual change of variable and by writing the noise and the Ornstein-Uhlenbeck process as finite sequence of 1D processes, we are able to prove there exists at least one invariant measure to (\ref{asnse4im}).

We remark that this fundamental ODE is different from the one used in the proof of existence and uniqueness. 
Let $f\in H$ and $m>1$ be given. 
 Consider
\begin{align}\label{asnse4im} du(t)=[-Au(t)-B(u(t),u(t))+\mathbf{C}u(t)+f]dt+\sum^m_{l=1}\sigma_le_ldL_l(t),
\end{align}
where operators $A$, $B$, $C$ are are as defined, $f\in H$, $L_1,\,L_2\,\cdots L_l$ are i.i.d. $\R$-valued symmetric $\beta$-stable process on a common probability space $(\Omega,\mathcal{F},\P)$, $\sigma$ is a bounded sequence of real numbers and $e_l$ is the complete orthonormal system of eigenfunctions on $H$.
\subsubsection{Auxiliary Ornstein-Uhlenback Process}
Let $(\tilde{L}(t),t\ge 0)$ be a L\'evy process that is an independent copy of $L$. Denote by $\bar{L}$ a L\'evy process on the whole real line by
\begin{align}
\bar{L}(t)
    \begin{cases}
        L(t),\quad t\ge 0\\
        \tilde{L}(-t),\quad t<0,
    \end{cases}
\end{align}
and by $\bar{\mathcal{F}}_t$ the filtration
\begin{align*} \bar{\mathcal{F}}_t=\sigma(\bar{L}(s),s<t),\quad t\in\R.
\end{align*}
Let $\alpha>0$ be given;  For each $l=1,\cdots,m$, let $z^0_l$ be the stationary (ergodic) solution of the one dimensional equation
\begin{align*}
    dz^0_l+(\lambda_l+\alpha) z^0_ldt=\sigma_ldL_l(t)
\end{align*}
so that 
\begin{align*}  z^0_l(t)=\int^t_{-\infty}e^{-(\lambda_l+\alpha)(t-s)}\sigma_ldL_l(s)
\end{align*}
Note that the integral above is well defined, since for any $p\in(1,\beta)$ with $\beta>1$ we have 
\begin{equation}\label{eq_fin}
\begin{aligned}
\mathbb E\left|z^0_l(t)\right|^p&=C_{p,\beta}\int_0^\infty e^{-p(\lambda_l+\alpha)(t-s)}\sigma_l^p\,ds\\
&=\dfrac{C_{p,\beta}\sigma_l^p}{p\left(\alpha+\lambda_l\right)}.
\end{aligned}
\end{equation}
More precisely, let 
\[z^0_l(t,s)=\int^t_{s}e^{-(\lambda_l+\alpha)(t-r)}\sigma_ldL_l(r)\,.\]
Then one can show directly evaluating integrals in the same way that 
\[\lim_{s\to-\infty}z_l(t,s)=z_l(t)\]
exists. 
Putting $z^0(t)=\sum^m_{l=1}z^0_l(t)e_l$ one has
\begin{align}\label{auxou}
    dz^0+(A+\alpha I) z^0\,dt=GdL(t)\,,
\end{align}
where $Ge_l=\sigma_le_l$, or
\[z^0(t)=\int_{-\infty}^t e^{-(t-s)(A+\alpha I)}GdL(s)\,.\]
 We have for any $s,t$ such that $-\infty<s<t<\infty$ 
\[z(t)=\int_{-\infty}^te^{-(t-s)\widehat A}GdL(s)=e^{-(t-s)\widehat A}z(s)+\int_s^te^{-(t-r)\widehat A}GdL(r)\,.\]
We need another lemma. 
\begin{lem}\label{dom_A}
We have 
\[\sup_{-1\le t\le 0}|Az(t)|^2<\infty\,.\]
\end{lem}
\begin{proof}    
Note first that the process $Z^0=Az^0$ is well defined and satisfies all the assumptions of Lemma 3.1 in \cite{snse1} with the process $L$ replaced by another L\'evy process $AL$. Therefore, we have 
\[\sup_{-1\le t\le 0}\left|Z^0(t)\right|^2<\infty\,.\]
Since $D(A)=D\left(\widehat A\right)$, one can repeat all arguments from the proof of Lemma 3.1. This yields
\[\sup_{-1\le t\le 0}\left|Az(t)\right|^2<\infty\,.\]
\end{proof}
Now, using the lemma above and Lemma 3.1 applied with $\delta=\frac12$ we find that the process $z$ is c\`ad\`ag in $V$ and   
\begin{align}\label{zzvaz}
    \sup_{-1\le t\le 0}\left(|z(t)|^2+|z(t)|^2_V+|Az(t)|^2\right)<\infty\quad\P\,\,\text{a.s.}\quad .
\end{align}
Using  equation (4.12) in \cite{MR2773026}, one can now choose $\alpha>0$ such that    
\begin{align}\label{z10}
    4\eta m\E|z_1(0)|\le \frac{\lambda_1}{4},
\end{align}
where $\lambda_1$ is the first eigenvalue of $A$, since $\E|z_1(0)|^p\to 0$ as $\alpha\to\infty.$

 From (\ref{z10}) and the Ergodic Theorem we obtain
\begin{align*} \lim_{t_0\to-\infty}\frac{1}{-1-t_0}\int^{-1}_{t_0}4\eta\sum^m_{l=1}|z_l(s)|ds=4\eta m\E|z_1(0)|<\frac{\lambda_1}{4}.
\end{align*}
Put $\gamma(t)=-\frac{\lambda_1}{2}+4\eta\sum^m_{l=1}|z_l(t)|$, we get
\begin{align}\label{gp1}
    \lim_{t_0\to-\infty}\frac{1}{-1-t_0}\int^{-1}_{t_0}\gamma(s)ds<-\frac{\lambda_1}{4}.
\end{align}
From this fact and by stationarity of $z_l$ we finally obtain
\begin{align}\label{gp2}   \lim_{t_0\to-\infty}e^{\int^{-1}_{t_0}\gamma(s)ds}=0\quad\P-\text{a.s.}\quad,
\end{align}

\begin{align}\label{gp3}   \sup_{t_0<-1}e^{\int^{-1}_{t_0}\gamma(s)ds}|z(t_0)|^2<\infty,\quad\P-\text{a.s.}\quad .
\end{align}

\begin{align}\label{gp4}
    \int^{-1}_{-\infty}e^{\int^{-1}_{\tau}\gamma(s)ds}(1+|z_l(\tau)|^2+|z_l(\tau)|^2_V+|z_l(\tau)|^2|z_l(\tau)|)d\tau<\infty,\quad\P-\text{a.s.}\quad. 
\end{align}
for all $1\le j$, $l\le m.$ Indeed, note for instance that for $t<0$,
\begin{align*} \frac{z_l(t)}{t}=\frac{z_l(0)}{t}-\frac{1}{t}(\alpha+A_l)\int^0_t z_l(s)ds+\frac{L_l(t)}{t},
\end{align*}
whence $\lim_{t\to-\infty}\frac{z_l(t)}{t}=0$ $\P$-a.s., which implies (\ref{gp2}) and (\ref{gp3}).
Consider the abstract SNSE
\begin{align*}   du+[Au+B(u)+\mathbf{C}u]dt=fdt+GdL(t)
\end{align*}
and the Ornstein-Uhlenback equation
\begin{align*}
    dz+(\hat{A}+\alpha I) z dt=GdL(t),
\end{align*}
where $L(t)=\sum^{m}_{l=1}e_l L_l(t)$.
We now use the change of variable $v(t)=u(t)-z(t).$ Then, by subtracting the Ornstein-Uhlenback equation from the abstract SNSE, we find that $v$ satisfies the equation
\begin{align}\label{odeim}
    \frac{d^+v}{dt}=-\nu Av(t)-\mathbf{C}v(t)-B(u,u)+f+\alpha z.
\end{align}

Recall the Poincare inequalities
\begin{align}\label{poincare}
    |u|^2_V&\ge\lambda_1|u|^2,\quad\forall\quad u\in V,\\
    |Au|^2&\ge\lambda_1|u|^2,\quad\forall\quad u\in D(A).
\end{align}
Let us note that  there exists $\eta>0$ such that
\begin{align}\label{bsum}
    |\langle B(u,e_l),u\rangle|\le\eta|u|^2,\quad u\in V, l=1,\cdots,m.
\end{align}
Then the following holds. 
\begin{prop}\label{est1}
    Let $\alpha>0$, $ v$ is a mild solution of (\ref{odeim}), there exist cosntants $c, c'>0$ depending only on $\lambda_1$ such that
\begin{align}\label{ineq1}
    \frac{1}{2}\frac{d^+}{dt}|v|^2+\frac{1}{2}|v|^2_V\le \left(-\frac{\lambda_1}{4}+2\eta\sum^m_{l=1}|z_l|\right)|v|^2+c|f|^2+c\alpha|z|^2+2\eta|z|^2\sum^m_{l=1}|z_l|.
\end{align}
\end{prop}

\begin{proof}
Let $\alpha> 0$ be given. Denote for simplicity by $z(t)$ the stationary Orstein Uhbleck process, corresponding to $\alpha$, introduced in earlier. Using the classical change of variable $v(t)=u(t)-z(t)$
 , the well known identity $\frac{1}{2}\partial_t|v(t)|^2=(v(t),v(t))$, and the antisymmetric term $(Cv,v) =0$ we have
\begin{align}\label{deq}    \frac{1}{2}\frac{d^+}{dt}|v|^2&=-\nu(Av,v)-\langle B(u,z),u\rangle+(\alpha z,v)+\langle f,v\rangle\\
    &\le -\nu |v|^2_V-\langle B(u,z),u\rangle+\alpha |z||v|+|f||v|.
\end{align}
By the definition of $z$ and assumptions (\ref{bsum}),
\begin{align*}
    \langle B(u,z),u\rangle&=\sum^m_{l=1}\langle B(u,e_l),u\rangle z_l\le\eta|u|^2\sum^m_{l=1}|z_l|\\
    &\le 2\eta|v|^2\sum^m_{l=1}|z_l|+2\eta|z|^2\sum^m_{l=1}|z_l|,
\end{align*}
and the inequalities
\begin{align*}
    \langle \alpha z, v\rangle = c\alpha |z|^2+c'|v|^2,
\end{align*}

\begin{align*}
    \langle f, v\rangle \le c |f|^2+c'|v|^2.
\end{align*}

For simplicity we take $\nu=1$.
Then via Young inequality, one can show that there exists $c, c'>0$ depending only on $\lambda_1$ such that
\begin{align*}
    \frac{1}{2}\frac{d^+}{dt}|v|^2+\frac{1}{2}|v|^2_V\le -\frac{1}{2}|v|^2+2\eta|v|^2\sum^m_{l=1}|z_l|+2\eta |z|^2\sum^m_{l=1}|z_l|+c|f|^2+2c'|v|^2+c\alpha |z|^2+2c|z|^2_V+c'|v|^2_V.
\end{align*}
So
\begin{align*}
    \frac{1}{2}\frac{d^+}{dt}|v|^2+\frac{1}{2}|v|^2_V\le (-\frac{\lambda_1}{4}+2\eta\sum^m_{l=1}|z_l|+2c')|v|^2+c|f|^2+c\alpha|z|^2+2\eta|z|^2\sum^m_{l=1}|z_l|.
\end{align*}
Hence one can find a constant $c, c'>0$ depending only on $\lambda_1$ for which the claim follows.
Moreover, Let $\gamma(t)$, and $p(t)$ are defined as :
\begin{align*}    p(t)=c|f|^2+c\alpha|z|^2+\eta|z|^2\sum^m_{l=1}|z_l(t)|,
\end{align*}
\begin{align*}   \gamma(t)=-\frac{\lambda_1}{2}+4\eta\sum^m_{l=1}|z_l(s)|,
\end{align*}
 we have
 \begin{align}\label{dineq}     \frac{1}{2}\frac{d^+}{dt}|v|^2+\frac{1}{2}|v|^2_V\le \frac{1}{2}\gamma(t)|v|^2+p(t).
 \end{align}
 \end{proof}
Temporarily disregard the $|v(t)|_V$ term, we have
\begin{align*}    \frac{d^+}{dt}|v(t)|^2\le\gamma(t)|v(t)|^2+2p(t)
\end{align*}
which implies
\begin{align}\label{ineq}
    |v(t)|^2\le |v(\tau)|^2 e^{\int^t_{\tau}\gamma(s)ds}+\int^t_{t_0}e^{\int^t_s\gamma(\xi)d\xi}2p(s)ds.
\end{align}
Now drop out the first term in (\ref{dineq}), integrate over $[\tau,t]$ we have 
\begin{align}\label{intineq}
    \int^t_{\tau}|v(s)|^2_Vds\le (\sup_{\tau\le s\le t}|v(s)|^2)\int^t_{\tau}\gamma(\xi)d\xi +\int^t_{\tau} 2p(s)ds.
\end{align}
Let us recall, that we proved the existence an uniqueness of solutions to the stochastic Naver-Stokes equation under the assumption that 
\begin{equation}\label{eq_v1}
\sum_{l=1}^\infty|\sigma_l|^{\beta}\lambda^{\beta/2}_l<\infty\,,
\end{equation}
and then the process $z(\cdot)$ has a c\`adl\`ag  version in $V=D(A^{1/2})$.

We will show that, under the above assumption, there exist at least one invariant measure for SNSE.
Let  $f\in H$, 
be given.
For an arbitrary real number $s$, $u(t,s)$, $t\ge s$, is the unique solution to the SNSE
\begin{align*}
    \begin{cases}
        d u(t)+A u(t)dt+B( u(t), u(t))dt+\mathbf{C} u(t)dt=fdt+dL(t),\\
         u(s)=0
    \end{cases}
\end{align*}
\begin{rmk}
    The space $D( A^{\delta})$ is compactly embedded into the space $H.$
\end{rmk}
\vspace{1cm}
 Consequently, if one prove that the process $ u(t,0), t\ge 1$ is bounded in probability as a process with values on $D( A^{\delta})$, one gets immediately the law $\mathcal{L}( u(t,0))$, $t\ge 1$ are tight on $H.$ This suffices the claim of existence of an invariant measure. More precisely, one proves in two steps.
 
\textbf{Step 1} Assuming that \eqref{eq_v1} holds we will prove an a priori bound in $H$. For any $\alpha\ge 0$, denote by $ z$ the stationary solution of
\begin{align*}
    d z_{\alpha}+(\hat{A}+\alpha I) z_{\alpha} dt=d\bar L(t),
\end{align*}
where
\begin{align}
     z_{\alpha}(t)= z(t)+e^{-(\hat{A}+\alpha)(t-s)}(z_{\alpha}-z(s))-\alpha\int^t_s e^{-(\hat{A}+\alpha)(t-s)}z(\sigma)d\sigma.
\end{align}
Let
\begin{align*}
     v_{\alpha}(t,s)= u(t,s)- z_{\alpha}(t),\quad t\ge s.
\end{align*}
Then $ v_{\alpha}(t)= v_{\alpha}(t,s)$, $t\ge s$ is the mild solution to
\begin{align*}
    \partial_t  v_{\alpha}(t)+\nu A v(t)dt+\mathbf{C}(v_{\alpha}(t)+z_{\alpha}(t))&=-B( v_{\alpha}(t)+ z_{\alpha}(t))+f+\alpha  z_{\alpha}(t),\quad t\ge s,\\
     v(s)=- z_{\alpha}(s)
\end{align*}
Following step 1 one has the following proposition
\begin{prop}\label{vaest}
    There exists $\alpha>0$ and a random variable $ \xi$ such that $\P$-a.s.
\begin{align}
    | v_{\alpha}(t,s)|&\le \xi\quad\forall\,t\in [-1,0]\quad\text{and all}\quad s\le -1,\label{vats}\\
    \int^0_{-1}| v_{\alpha}(t,s)|^2_{V}ds&< \xi\quad\forall\,\,t\in [-1,0]\,\,\text{and all}\,\,s\le -1.\label{vatsi}
\end{align}
\end{prop}
\begin{proof}
    In view of inequality \ref{ineq}, one obtains
\begin{align}\label{va}
    | v_{\alpha}(t,s)|^2\le |v(s)|^2 e^{\int^{t}_s -\frac{\lambda_1}{2}+4\eta\sum^m_{l=1}|z_l(\xi)|d\xi}+\int^t_s e^{\int^t_r\gamma(\xi)d\xi}2p(r)dr.
\end{align}
Based on the earlier discussion, the first term is finite; The second term is also finite under the assumption (\ref{zzvaz}).

We now use the ergodic properties of $z.$ Since $z_{\alpha}(t)$, $-\infty<t<\infty$, is an ergodic process which is supported by $D(A^{\delta})\subset    \mathbb{L}^4(\mathbb{S}^2)$. Then by the
 Marcinkiewicz strong law of large number, we have $\P$ a.s. that
 and by Prop 8.4 \cite{MR2584982} that
\begin{align*}
    \lim_{s\to-\infty}\frac{1}{-1-s}\int^{-1}_{s}4\eta\sum^m_{l=1}|z_l(\sigma)|d\sigma=4\eta m\E|z_1(0)|<\frac{\lambda_1}{4}.
\end{align*}

The existence and uniqueness of invariant measure for the OU equation driven by L\'evy process is well-known \cite{MR900115}.

Let $\mu_{\alpha}$ be the unique invariant measure of L\'evy type. It is easy to see that
\begin{align*}
    \lim_{\alpha\to\infty}\int_{V}4\eta\sum^m_{l=1}|z_l(s)|\mu_{\alpha}(dz)=0.
\end{align*}
Then for sufficiently large random $s_0>0$ and $s<-s_0$
\begin{align}\label{expb}
    e^{\int^{t}_s -\frac{\lambda_1}{2}+4\eta\sum^m_{l=1}|z_l(\xi)|d\xi}\le e^{-\frac{\lambda_1}{4}(t-s)}.
\end{align}

To complete the proof this proposition we need the following Lemma.
\begin{lem}\label{estz}
    Assume that  $X$ is a stationary process taking values in a Banach space $B$. Moreover, assume that for a certain $p>0$ we have 
    \[\E\sup_{t\in[-1,0]}|X(t)|_B^p<\infty\,.\]
    Then for every $\kappa>0$ such that $\kappa p>1$ there exists a random variable $ \xi$ such that $\P$ a.s.
    \begin{align*}
|X(t)|_B\le \xi+2^\kappa |t|^{\kappa},
    \end{align*}
for all $t\le 0$. 
\end{lem}
\begin{proof} 

Let $\eta_n=\sup_{-n\le s\le -n+1}|X(s)|_B$, $n=0,1,\ldots$ Then by stationarity 
    \begin{align*}
    \E\eta^p_n=\E\eta^p<\infty .   &
    \end{align*}
Therefore,
\begin{align}
    \P(\eta_n\ge n^{\kappa})\le\frac{\E\eta^p_1}{n^{\kappa p}}.
\end{align}

If $\kappa p>1$, then $\sum^{\infty}_{n=1}\P(\eta_n\ge n^{\kappa})<\infty$, and by the Borel Cantelli lemma, $\P$-a.s., for any sufficiently large $n$,
\begin{align*}
    \eta_n\le n^{\kappa}.
\end{align*}
That is, for every $\omega$, there exists  $N(\omega)$ such that 
\begin{align*}
    \eta_n(\omega)\le n^{\kappa},\quad\text{for}\quad n>N(\omega). \end{align*}
    Therefore, for $t\in[-n,-n+1]$ we have 
    \begin{align*}|X(t,\omega)|_B\le \eta_n(\omega)&\le\eta_n(\omega) I_{n\le n(\omega)}+n^\kappa I_{n>N(\omega)}\\
    &\le\eta_n(\omega)I_{n\le N(\omega)}+2^\kappa |t|^\kappa.
    \end{align*}
Since $\P(N<\infty)=1$ the random variable 
\[\xi(\omega)=\max_{n\le N(\omega)}\eta_n(\omega)\]
is finite $\P$-a.s. and the Lemma follows.
\end{proof}
With the aid of this lemma, combine with equations (\ref{vats}), (\ref{va}) and (\ref{expb}). We deduce the claim in Proposition \ref{vaest}. Moreover, via an apriori estimate about $\int^T_0| v(t)|^2_Vdt$, that is (\ref{intineq}) the inequality (\ref{vatsi}) follows.
\end{proof}

\textbf{Step 2} \emph{Measure support.}  We now generalise Proposition \ref{est1} by proving regularizing property of equation (\ref{odeim}) via deducing a priori estimate in $D( A^{\delta})$ for some $\delta>0.$
This allows us to establish support of invariant measure.
\begin{prop}\label{est3}
    For any $\delta\in [0,\frac{1}{2}]$, there exists $C=C(\delta)$ such that for any mild solution $v(\cdot)$ of (\ref{odeim}), one has
\begin{align}\label{est1spt}
    | A^{\delta} v(t)|^2\le e^{C\int^t_0| v(s)|^2| A^{\frac{1}{2}} v(s)|^2ds}|A^{\delta} v(0)|^2+C\int^t_0 e^{C\int^t_{\sigma}| v(s)|^2| A^{\frac{1}{2}} v(s)|^2ds}(| A^{\delta+\frac{1}{2}}f|^2+|z(\sigma)|^2+| A^{\frac{1+2\delta}{4}}z(\sigma)|^4)d\sigma.
\end{align}
\end{prop}
\begin{proof}
    Multiply (\ref{odeim}) by $ A^{2\delta} v$ and integrating over $\mathbb{S}^2$, one finds that
\begin{align}\label{eq2}
&\    \frac{1}{2}\partial_t| A^{\delta} v(t)|^2+| A^{\frac{1}{2}+\delta} v(t)|^2+(\mathbf{C} v(t), A^{2\delta} v(t))\notag\\
=&\ -b( v(t)+ z_{\alpha}(t), v(t)+ z_{\alpha}(t), A^{2\delta} v(t))+\alpha( A^{\delta} z_{\alpha}(t), A^{\delta} v(t))+\langle  A^{\delta}f, A^{\delta} v(t)).
\end{align}
From Lemma 2.4 in \cite{snse1} it is clear that
\begin{align*}
    ( \mathbf{C} v,  A^{2\delta} v)=0.
\end{align*}

To complete the proof we need to estimate the terms  $b(v+z,v+z,A^{2\delta}v)$, $\alpha\langle A^{2\delta}v, z\rangle$, $\langle A^{2\delta}v, f\rangle$

Using Young inequality with $ab=\sqrt{\frac{\nu}{10}}a \sqrt{\frac{10}{\nu}}b,\,p=2$, we have

\begin{align*}
    \alpha|\langle A^{2\delta}v, z\rangle|\le\frac{\nu}{6}|A^{\delta+\frac{1}{2}}v|^2+\frac{3\alpha^2}{2\nu}|z|^2
\end{align*}

\begin{align*}
    |\langle A^{2\delta}v, f\rangle|\le\frac{\nu}{6}|A^{\delta+\frac{1}{2}}v|^2+\frac{3}{2\nu}|f|^2
\end{align*}

Finally, following the method of deriving (15.4.12) as in \cite{MR1417491}, one can show that, for any $\nu>0$, there exists a $K(\nu)$ such that
\begin{align}\label{bvzadelv}
|\langle A^{2\delta}v,B( v+ z, v+ z)\rangle|&=|b( v+ z, v+ z, A^{2\delta} v)|\\
&\le \frac{\nu}{6} | A^{\delta+\frac{1}{2}} v|^2+K(\nu)(| v|^2| A^{1/2} v|^2+| A^{\frac{1+2\delta}{4}} z|^4).
\end{align}
Combing the above estimates, we have
\begin{align*}
&\ \frac{1}{2}\partial_t| A^{\delta} v(t)|^2+(1-3\frac{\nu}{6})|A^{\frac{1}{2}+\delta} v(t)|^2\\
\le &\ K(\nu)|v(t)|^2| A^{\frac{1}{2}} v(t)|^2+K(\nu)|A^{\frac{1+2\delta}{4}} z(t)|^4+\frac{3\alpha^2}{2\nu}| z(t)|^2+\frac{3}{2\nu}|f|^2.
\end{align*}
Therefore, invoking Gronwall, it follows that
\begin{align}
    | A^{\delta} v(t)|^2&\le e^{K(\nu)\int^t_0| v(s)|^2| A^{\frac{1}{2}} v(s)|^2 ds}| A^{\delta} v(0)|^2\\&+\int^t_0 e^{K(\nu)\int^t_{\sigma}| v(s)|^2| A^{\frac{1}{2}} v(s)|^2ds}\left(\frac{3\alpha^2}{2\nu}|z|^2+K(\nu)|A^{\frac{1+2\delta}{4}}|^4+\frac{3}{2\nu}|f|^2\right)d\sigma.
\end{align}
\end{proof}
To complete the proof of invariant measure. It follows from Proposition \ref{est3} that for any $t\le -1\le r\le 0$,
\begin{align*}
&\ | A^{\delta} v_{\alpha}(0,t)|^2\\
=&\ e^{K(\nu)\int^0_r| v_{\alpha}(s,t)|^2| A^{\frac{1}{2}} v_{\alpha}(s,t)|^2 ds}| A^{\delta}v_{\alpha}(r,t)|^2\\&+\int^0_re^{K(\nu)\int^0_{\sigma}|v_{\alpha}(s,t)|^2| A^{\frac{1}{2}}v_{\alpha}(s,t)|^2 ds}(\frac{3\alpha^2}{\nu}|z|^2+K(\nu)|A^{\frac{1+2\delta}{4}}|^4+\frac{3}{2\nu}|f|^2)d\sigma\\
\le &\ e^{K(\nu)[\sup_{-1\le s\le 0}| v_{\alpha}(s,t)|^2]\int^0_{-1}| A^{\frac{1}{2}} v_{\alpha}(s,t)|^2 ds}\times\left[| A^{\delta} v_{\alpha}(r,t)|^2+\frac{3\alpha^2}{2\nu}|z|^2+K(\nu)|A^{\frac{1+2\delta}{4}}|^4+\frac{3}{\nu}|f|^2d\sigma\right].
\end{align*}
Consequently, integrating the above over the interval $[-1,0]$, one gets for $t\le -1$ that
\begin{align*}
&\    | A^{\delta} v_{\alpha}(0,t)|^2\\
    \le &\ e^{K(\nu)[\sup_{-1\le s\le 0}| v_{\alpha}(s,t)|^2]\int^0_{-1}| A^{\frac{1}{2}} v_{\alpha}(s,t)|^2 ds}\\&\times\left[| A^{\frac{1}{2}} v_{\alpha}(r,t)|+\frac{3\alpha^2}{2\nu}|z|^2+K(\nu)|A^{\frac{1+2\delta}{4}}|^4+\frac{3}{2\nu}|f|^2d\sigma\right].
\end{align*}
By Proposition \ref{vaest} there exists a random variable $\eta$ such that $\P$ a.s.
\begin{align}
    | A^{\delta} v_{\alpha}(0,t)|\le  \xi,\quad\forall\,\,t\le -1.
\end{align}
Moreover,
\begin{align*}
    | A^{\delta} u(0,t)|\le | A^{\delta} v_{\alpha}(0,t)|+| A^{\delta} z_{\alpha}(0)|.
\end{align*}
Since $ z_{\alpha}(0)$ takes value in $D( A^{\delta})$ there exists another random variable $\zeta$ such that $\P$ a.s.
\begin{align}\label{adelu}
    | A^{\delta} u(0,t)|\le\zeta\quad\forall\,\,t\le -1.
\end{align}
So $u(0,t)$ is bounded in probability in the space $D( A^{\delta})$ for some $\delta>0$ satisfies $\sum_{l\ge 1}|\sigma_l|^{\beta}\lambda^{\beta\delta}_l<\infty$:
\begin{align*} \forall\,\,\epsilon>0\,\,\exists\,\,R>0\,\,\forall\,\,t\ge 0\quad\P(| u(0,t,u_0)|\ge R)<\epsilon.
\end{align*}
Now Let $ u_0$ be fixed and let $\nu_{t, u_0}$ be the law of $u(t, u_0).$ Set
\begin{align*}
  \mu_T=\frac{1}{T}\int^T_0\nu_{t, u_0}dt.
\end{align*}
Let $B_R=\{ x\in D( A^{\delta});| A^\delta x|\le R\}$, equation (\ref{adelu}) implies for $p\in(1,\beta)$
\begin{align*}
    \mu_T(B^c_R)&\le \frac{1}{TR^p}\int^T_0\E|A^{\delta} u(0,t, u_0)|^p\,dt\\
    &\le\frac{1}{TR^p}T\mathbb E\zeta^p=\frac{\mathbb E\zeta^p}{R^p}.
\end{align*}
We have that, for any $\epsilon>0$, $\mu_T(B_R)=1-\epsilon$ for sufficient large $R$. Hence $\mu_T$ is tight and its limit is an invariant probability measure of the solution $u$ of equation (\ref{asnse4}), by Corollary \ref{kbcor}.
Moreover, the support of the invariant measure is in $D(A^{1/2})$. 

Combine with the markov feller properties proved for $u$ earlier, the solution $u$ to equation (\ref{asnse4}) admits at least one invariant measure and is supported in $D(A^{1/2})$. Hence, Theorem \ref{existenceim} is proved.
\begin{rmk}
	To prove uniqueness of invariant measure, one needs irreducibility and strong feller properties of solution semigroup. However, there is a trade-off between well-posedness of the SNSE and the strong Feller property (see the publication \cite{Dong:2012uq}).
\end{rmk}
\bibliographystyle{plain}
\bibliography{l'sthesis}
\end{document}